\documentclass{amsproc}

\newtheorem{theorem}{Theorem}[section]
\newtheorem{lemma}[theorem]{Lemma}
\newtheorem{corollary}[theorem]{Corollary}

\theoremstyle{definition}
\newtheorem{definition}[theorem]{Definition}

\newtheorem{mainexample}[theorem]{Main Example}

\theoremstyle{remark}
\newtheorem{remark}[theorem]{Remark}

\numberwithin{equation}{section}

\newcommand{\blankbox}[2]{%
  \parbox{\columnwidth}{\centering
    \setlength{\fboxsep}{0pt}%
    \fbox{\raisebox{0pt}[#2]{\hspace{#1}}}%
  }%
}

\def\Z{\mathbb{Z}}

\def\Q{\mathbb{Q}}
\def\R{\mathbb{R}}
\def\C{\mathbb{C}}
\def\H{\mathbb{H}}

\def\iso{\cong}

\def\t{\mathfrak{t}}

\def\inv{^{-1}}
\newcommand{\excise}[1]{}

\input xy
\xyoption{all}

\begin{document}

\title{Weighted hyperprojective spaces and homotopy invariance in orbifold cohomology}
\date{\today}
\author{Rebecca Goldin}
\address{Mathematical Sciences, George Mason University, Fairfax, VA 22030}
\email{rgoldin@math.gmu.edu}
\thanks{The author was supported in part by NSF Grants \#0305128 and \#0606869.}


\date{May 31, 2007}



\begin{abstract}
We show that Chen-Ruan cohomology is a homotopy invariant in certain cases. We introduce the notion of a {\em $T$-representation homotopy} which is a stringent form of homotopy under which Chen-Ruan cohomology is invariant. We show that while hyperk\"ahler quotients of $T^*\C^{n+1}$ by $S^1$  (here termed {\em  weighted hyperprojective spaces}) are homotopy equivalent to weighted projective spaces, they are not $S^1$-representation homotopic. Indeed, we show that their Chen-Ruan cohomology rings (over $\Q$) are distinct.
\end{abstract}

\maketitle

\section*{Introduction}

In the toric topology conference held at Osaka University in 2006, the author spoke about joint results in \cite{GH} with M. Harada. In that paper, we compute the Chen-Ruan cohomology ring over $\Q$ defined in \cite{CR} for hypertoric varieties (similar results were also found by \cite{JT} using entirely different techniques). Our results relied heavily on methods developed in \cite{GHK} and also on results due to H. Konno \cite{Ko} and T. Hausel and B. Sturmfels \cite{HS} about the topology of these varieties. In this note, we review these and some of the results discussed at the conference. We then apply them to describe combinatorially the Chen-Ruan cohomology of the case of a hyperk\"ahler reduction by a linear and hyperk\"ahler $S^1$ action on a vector space, which we term a {\em weighted hyperprojective space}.

A secondary goal of this paper is to explain a sense in which Chen-Ruan cohomology is homotopy invariant, and to illustrate that Chen-Ruan cohomology fails to be invariant under a naive notion of homotopy equivalence for global quotients. In general, homotopy invariance is tricky, since orbifolds are only locally defined as quotients. Any map between orbifolds must not only commute with the local group structure, but also preserve global topological properties (see \cite{Ch} for a notion of homotopy groups, \cite{Mo} for a groupoid treatment). Depending on the presentation of these orbifolds, it may be difficult to identify ``good" maps between orbifolds, in the sese of \cite{CR1}. However, even in the global quotient case, equivariant homotopies do not guarantee that the Chen-Ruan cohomology is preserved. We introduce the notion of a {\em $T$-representation homotopy} and show that Chen-Ruan cohomology (and inertial cohomology) is preserved under such a homotopy.

The case of weighted hyperprojective spaces is particularly useful to describe the difficulty involved with homotopy equivalence. These spaces retract to weighted projective spaces \cite{BD}, yet their Chen-Ruan cohomologies are not equal. This is the case even though this homotopy can  be described ``upstairs" on manifolds, before taking a quotient by $S^1$. The existence of the homotopy implies that the (ordinary or equivariant) cohomology ring -- even the integral cohomology ring -- of a hyperprojective space is known to be the cohomology of the corresponding weighted projective space. While the {\em groups} occurring in Chen-Ruan cohomology are homotopy invariant because they are cohomology groups, the twisted product introduced in \cite{CR} is not invariant under (ordinary or equivariant) homotopy.

\excise{Hypertoric varieties are homotopic to their core, as shown in  and described in this proceedings by REFERENCE HERE NICK'S PAPER IN THIS PROCEEDINGS. In the case of a weighted hyperprojective space $\mathfrak{M}$, the core $C(\mathfrak{M})$ is a weighted projective space, and the underlying space of the weighted hyperprojective space retracts to this core.}

Weighted projective spaces may be described by the quotient of $S^{2n+1}$ by an appropriate $S^1$ action. Let $S^1$ act on $S^{2n+1}=\{(z_0,\dots, z_n)\in \C^{n+1}|\ |z_i|^2=1\}$ by $t\cdot (z_0,\dots, z_n)= (t^{b_0}z_0,\dots,t^{b_n} z_n)$ for some nonnegative integers $b_i$. Then $\C P^n_{(b)}:= S^{2n+1}/S^1$ is the quotient space.  The (ordinary) cohomology $H^*(\C P^n_{(b)};\Z)$ is the cohomology of the underlying topological space, which was computed by Kawasaki \cite{Ka}. In contrast, the cohomology of the {\em orbifold} $[\C P^{n}_{(b)}]$ is by definition $H^*([\C P^n_{(b)}];\Z) := H_{S^1}(S^{2n+1};\Z)$. \excise{The brackets are used to emphasize the added information of the orbifold with its presentation as a quotient.} The cohomology of the underlying space and that of the orbifold are isomorphic over $\Q$, but not over $\Z$.  To describe this cohomology, one can use symplectic techniques and Kirwan surjectivity (see \cite{Ho} for an explicit computation).

In sharp contrast to ordinary cohomology, the Chen-Ruan cohomology of $[\mathfrak{M}]$ is not isomorphic to that of its core $C(\mathfrak{M})$. Chen-Ruan cohomology is not a homotopy invariant, even over $\Q$. In Section~\ref{se:computationCRweightedhyperprojective} we explicitly compute this twisted cohomology and note its disagreement with the Chen-Ruan cohomology of weighted projective spaces. In Section~\ref{se:homotopy} we show that Chen-Ruan cohomology is invariant under $T$-representation homotopy, a stringent form of homotopy.

This article is organized as follows: inertial cohomology is defined in Section~\ref{se:inertial}; its relation to Chen-Ruan cohomology of symplectic and hyperk\"ahler (abelian) quotients is described in Section~\ref{CRcohomology}.\excise{ the construction of hypertoric varieties is reviewed in Section~\ref{se:hypertoric}; its salient properties are discussed in Section~\ref{se:core}.} Finally, the main results of the paper are found in Sections~\ref{se:computationCRweightedhyperprojective} and \ref{se:homotopy}.

It was a particular pleasure to participate in the International Conference on Toric Topology last spring in Osaka. The talks were of unusually high quality, and the conference could not have been better organized. I would especially like to thank
Megumi Harada, Taras Panov, Yael Karshon and Mikiya Masuda for providing such an inviting environment for mathematics.

\section{Inertial cohomology}\label{se:inertial}

Here we briefly describe the {\em inertial cohomology} of a stably almost complex manifold $Y$ with an action by a torus $T$, introduced in \cite{GHK}. When $Y$ has a locally free action by $T$ -- i.e. $T$ acts with finite isotropy on $Y$ --  the inertial cohomology $NH_T^{*,\diamond}(Y)$ equals (as a ring) the Chen-Ruan cohomology of $[Y/T]$, as defined in \cite{CR}. On the other hand, when $Y$ is Hamiltonian or hyperhamiltonian, its inertial cohomology has special properties that lend themselves to easy computations. In some circumstances the inertial cohomology of a (hyper)hamiltonian space $Y$ surjects onto that of a level set of the moment map, which in turn is isomorphic the Chen-Ruan cohomology of the orbifold given by the level set quotiented by $T$.

\excise{ Inertial cohomology has the advantage of being defined for spaces $Y$ that do not have locally free $T$ actions, such as those occurring in the Hamiltonian and hyperhamiltonian contexts. Even more importantly, under the appropriate conditions, an inclusion of spaces induces a restriction map in inertial cohomology; for a Hamiltonian or hyperhamiltonian $T$-space, inertial cohomology  restricts to the inertial cohomology of the level set of the relevant moment map, which is isomorphic to the Chen-Ruan cohomology of the corresponding symplectic or hyperk\"ahler quotient. This provides new methods for computing the Chen-Ruan cohomology of these orbifolds.
}
\begin{remark} The space $Y$ need not be honestly almost complex; the inertial cohomology of $S^1$ acting appropriately on $Y=S^3$ equals the Chen-Ruan orbifold cohomology of the quotient, a (possibly weighted) projective space $S^3/S^1$. However, when $Y$ is stably almost complex, the normal bundles $\nu(Y^{g,h}\subset Y^g)$ are honestly complex, for all $g,h\in T$.
\end{remark}

\begin{remark}
Clearly not all orbifolds may be expressed as a quotient $[Y/T]$ with $T$ an abelian Lie group. However, the set of orbifolds that can be presented this way include spaces that are not quotients by finite groups, such as weighted projective spaces.
\end{remark}
\excise{In the case that $Y$ is a vector space, the inertial cohomology simplifies further. Theorems~\ref{th:GHK} and \ref{th:GH} specify how the techniques introduced in this section (and, in more detail, in \cite{GHK} and \cite{GH}) apply to the cases of toric varieties and hypertoric varieties, respectively. These spaces are symplectic and hyperk\"ahler, respectively, quotients of vector spaces by linear torus actions.
}

As a module over $H_T^*(pt;\Z)$, inertial cohomology is defined by
$$
NH_T^{*,\diamond}(Y;\Z):=\bigoplus_{g\in T} H_T^*(Y^g;\Z) = \bigoplus_{g\in T} H^*([Y^g/T];\Z),
$$ where $Y^g$ is the fixed point set of $g$ acting on $Y$, and $H_T^*(Y)$ is the {\em equivariant cohomology} of $Y$.\footnote{Let $ET$ be a contractible space with a free $T$ action on it, and let $Y_T :=(Y\times ET)/T$ be the associated Borel homotopy quotient. Then $H_T^*(Y;\Z) := H^*(Y_T;\Z)$ is the singular cohomology of $Y_T$.} The inertial cohomology is defined analogously for other coefficient rings, and we suppress the coefficient ring when it is irrelevant. \excise{The grading on this vector space is not the same as that of equivariant cohomology; it is graded by real numbers, rather than integers (see \cite{GHK}). The grading $\diamond$ is in the group elements; for any $g,h\in T$, the product of elements $a\in NH_T^{*,g}(Y)$ and $b\in NH_T^{*,h}(Y)$ is $a\smile b \in NH_T^{*,gh}(Y).$} See \cite{GHK} for details on the grading.

The product on $NH_T^{*,\diamond}(Y)$ is defined as follows. Choose $g,h\in T$, and let $H = \langle g,h\rangle$ be the subgroup they generate. Then $Y^H$ is a submanifold of
$Y$, and the normal bundle \(\nu(Y^H)\) of $Y^H$ in $Y$ is naturally
equipped with an $H$-action on the fibers. Let $X$ be a connected component of $Y^H$
and $\nu(X):= \nu(Y^H)|_X$ be the restriction of the normal bundle to
$X$. We may decompose \(\nu(X)\) into isotypic components with respect
to the $H$-action:
$$
\nu(X) = \bigoplus_{\lambda\in\hat{H}} I_\lambda,
$$ where $\hat{H}$ denotes the character group of $H$.

\begin{definition}\label{def:logweight}
Let $\lambda \in \hat{H}$ and $t \in H$.
The {\bf logweight of $t$ with respect to $\lambda$}, denoted $a_{\lambda}(t)$,
is the real number in $[0,1)$
such that $\lambda(t) = e^{2\pi i
a_\lambda(t)}$.
\end{definition}

We need one other ingredient before defining the product.
\begin{definition}\label{def:obstructionbundle}
The {\bf obstruction bundle} is a vector bundle over each component
$X$ of $Y^{H}$ given by
$$ E|_X := \bigoplus_{\lambda\in \hat{H} \atop
a_\lambda(g)+a_\lambda(h)+a_\lambda((gh)^{-1})=2} I_\lambda.
$$
We write $E\rightarrow Y^H$ to denote the union $E|_X$ over all connected
components.\footnote{Note that the sum $a_\lambda(g)+a_\lambda(h)+a_\lambda((gh)^{-1})$ is always 0,1, or 2, and it is constant on a connected component of $Y^H$.} The dimension may vary over components. The {\bf virtual
fundamental class} $\epsilon\in H_T^*(Y^H)$ is given by
$$
\epsilon := \sum_{X \in \pi_0(Y^H)} e(E|_X)\in H_T^*(Y^H),
$$
where the sum is over connected components $X$ of $Y^H$ and $e(E|_X)$ is the $T$-equivariant Euler class of $E|_X$, considered as an element of $H_T^*(X)$.
\end{definition}
Finally, let $e_1: Y^H \rightarrow Y^{g}$ , $e_2: Y^H\rightarrow Y^h$ and $\overline{e}_3:Y^H \rightarrow Y^{gh}$ denote the natural inclusions. These induce pullbacks $e_1^*:H_T^*(Y^{g})\rightarrow H_T^*(Y^H)$ and $e_2^*:H_T^*(Y^{h})\rightarrow H_T^*(Y^H)$
 and the pushforward $(\overline{e}_3)_*: H_T^*(Y^H)\rightarrow H^*_T(Y^{gh})$.
Let $a\in NH_T^{*,g}(Y)$ and $b\in NH_T^{*,h}(Y)$ be homogeneous classes in $\diamond$. Then we define
\begin{equation}\label{eq:inertialproduct}
a\smile b:= (\overline{e}_3)_*(e_1^*(a)\cdot e_2^*(b)\cdot \epsilon) 
\end{equation}
where the products on the right hand side are computed in the usual product structure of $H^*_T(Y^H)$. Note that the result lies in $NH^{*,gh}_T(Y)$.  By linear extension, the product is defined for any two classes $a,b\in NH^{*,\diamond}_T(Y)$.

It is immediately clear from this definition that, for any subgroup $\Gamma\subset T$, the product structure restricts to $NH^{*,\Gamma}_T(Y):=\oplus_{g\in\Gamma} NH^{*,g}_T(Y)$, making this set a ring as well. We call this the {\bf $\Gamma$-subring of the inertial cohomology}.

\begin{remark}
In \cite{GHK} we also introduce the product $\star$ on $NH_T^{*,\diamond}(Y)$. In the case that $Y$ is a Hamiltonian $T$ space, $\star$ and $\smile$ coincide. The $\star$ product has some computational advantages and makes associativity all but obvious. The corresponding combinatorics are used extensively in \cite{GHK} and in \cite{GH} to describe the cohomology of the corresponding quotient spaces as the inertial cohomology of $Y$ modulo an ideal, all of which can be described using $\star$.

However, the $\star$ product obscures the ring map to the inertial cohomology of a level set $L$ of the hyperk\"ahler moment map. Indeed, the product is 0 on the inertial cohomology of any space with no fixed points. In particular, it does not coincide with the product introduced by Chen and Ruan on the quotient orbifold.
\end{remark}

\excise{
As we shall see in Section~\ref{CRcohomology}, the Chen-Ruan cohomology of a hypertoric variety can be computed as a quotient of the inertial cohomology of a vector spaces under a particular $T$ action by a certain (computable) ideal.} In the case of an $S^1$-action, this inertial cohomology can be computed as follows. We follow notation introduced in \cite{Ho}.

\begin{mainexample}\label{mex:inertial} Let $S^1$ act on $\C^{n+1}\oplus \C^{n+1}$ by
$$
t\cdot (z_0,\dots, z_n,w_0,\dots, w_n) = (t^{b_0}z_0,\dots, t^{b_n}z_n,t^{-b_0}w_0,\dots, t^{-b_n}w_n),
$$ where $b_i\geq 0.$ Let $\Gamma$ be the group generated by the finite stabilizers occurring under this action. Then $\Gamma\iso \Z/\ell\Z$ with $\ell:=lcm(b_0,\dots, b_n)$, and the $\Gamma$-subring of the inertial cohomology is given by
\begin{align*}
NH_{S^1}^{*,\Gamma}(\C^{n+1}\oplus \C^{n+1};\Z) = \Z[u,\alpha_0,\dots, \alpha_{\ell-1}]/\mathcal{I}, \quad\mbox{with}\\
\mathcal{I}=\langle\alpha_g\smile\alpha_h-
\alpha_{[g+h]}\prod_{i=0}^n (b_iu)^{\left([b_ig]+[b_ih]-[b_i(g+h)]\right)/\ell}\\
\cdot (-b_iu)^{\left([-b_ig]+[-b_ih]-[-b_i(g+h)]\right)/\ell}\rangle
\end{align*}
for all $g,h\in \Gamma$, where $[m]$ is the smallest nonnegative integer congruent $m$ modulo $\ell$. Note that, for any $i$, these exponents are equal when $[b_i(g+h)]=0,$ and otherwise only one of the two exponents will be nonzero.
\end{mainexample}
\begin{proof}
Note that the group $\Gamma$ is given by the $\ell^{th}$ roots of unity, where $\ell=lcm(b_0,\dots, b_n)$. We identify $\Gamma\iso \Z/\ell\Z$ with these roots. Then as a module over $H_{S^1}^*(pt;\Z)$,
$$
NH_{S^1}^{*,\Gamma}(\C^{n+1}\oplus \C^{n+1};\Z) = \bigoplus_{g\in \Gamma} H_{S^1}^*((\C^{n+1}\oplus \C^{n+1})^g;\Z).
$$
 For any $g\in \Gamma$, we note that $(\C^{n+1}\oplus \C^{n+1})^g = \oplus_{i\in S_g} \C_i\oplus \C_{n+1+i},$ where $S_g = \{i\in \{0,\dots, n\}|\ b_i g = 0 \mbox{ in } \Z/\ell\Z\}$ and $\C_i$ indicates the $i^{th}$ copy of $\C$ in $\C^{n+1}=\oplus_{i=0}^n\C_i$. In particular, all the fixed points sets are equivariantly homotopic to a point. Thus $H_{S^1}^*((\C^{n+1}\oplus \C^{n+1})^g;\Z)\cong H_{S^1}^*(pt)\cong \Z[u]$ for each $g\in \Gamma$. Thus as a $\Z[u]$-module, $NH_{S^1}^{*,\diamond}(\C^{n+1}\oplus \C^{n+1})$ is free with one generator for each element in $\Gamma$. We denote the $\Z[u]$-module generators by $\alpha_0,\dots, \alpha_{\ell-1}$. In other words, $\alpha_{[g]}$ is the symbol we use to denote ``1" in $H_{S^1}^*((\C^{n+1}\oplus \C^{n+1})^g)$.

A fiber of the normal bundle to the $g$-fixed point set is isomorphic to $\oplus_{i\in (S_g)^c} \C_i\oplus\C_{n+1+i}.$ For any $g,h\in \Gamma$, the obstruction bundle restricted to $(\C^{n+1}\oplus \C^{n+1})^{g,h}$ consists of lines $\C_i$ such that
\begin{equation}\label{eq:sumis2}
\frac{[b_ig]}{\ell}+\frac{[b_ih]}{\ell}+\frac{[b_i([-(g+h)])]}{\ell}=2,
\end{equation} \excise{where $[\ ]$ denotes the smallest nonnegative integer modulo $\ell$,} as well as those $\C_{n+1+i}$ such that
\begin{equation}\label{eq:sumis2'}
\frac{[-b_ig]}{\ell}+\frac{[-b_ih]}{\ell}+\frac{[b_i([(g+h)])]}{\ell}=2.
\end{equation}
Using the language of logweights, $\frac{[b_ig]}{\ell} = a_{b_i}(g)$. These are the lines whose equivariant Euler classes contribute to the virtual fundamental class. Notice that in order for the sum to be 2, the first two summands must have sum strictly greater than 1. In particular, none of the summands are 0. Note also that when $[bg]\neq 0$, we have $[-b_ig]=\ell-[b_ig]$. This implies that either Equation~(\ref{eq:sumis2}) or (\ref{eq:sumis2'}) holds, but not both (and possibly neither).

 On the other hand, the fiber of the normal bundle to $(\C^{n+1}\oplus \C^{n+1})^{g,h}$ in $(\C^{n+1}\oplus \C^{n+1})^{g+h}$ consists of those pairs $\C_i\oplus \C_{i+n+1}$ that are fixed by $g+h$ but not by both $g$ and $h$. This implies $g$ and $h$ satisfy
\begin{equation}\label{eq:pushforwardcontribution}
[b_ig]\neq 0, \quad [b_ih]\neq 0,\quad [b_i[g+h]]=0, \mbox{ implying } [b_ig]+[b_ih]=\ell.
\end{equation}
These are the planes contributing to the pushforward map $(\overline{e}_3)^*$.
 We let $R_{g,h}:= \{i\in\{0,\dots, n\}| \C_i\mbox{ occurs in }E\}$ and $R'_{g,h}:=\{i\in\{0,\dots, n\}| \C_{n+1+i}\mbox{ occurs in }E\}$. Then $R_{g,h}\cup R'_{g,h}$ are the set of indices that occur in the obstruction bundle. Let $R''_{g,h}$ be the set of $i\in \{0,\dots, n\}$ satisfying Equations (\ref{eq:pushforwardcontribution}). Note that $(R_{g,h}\cup R'_{g,h})\cap R''_{g,h}=\emptyset$.

For any pair $g,h$, the equivariant Euler class of the obstruction bundle is $\prod_{j\in R_{g,h},k\in R'_{g,h}} (b_ju)(-b_ku)$. On the other hand, the pushforward in Formula (\ref{eq:inertialproduct}) contributes\footnote{Note that in general $[b_i(-(g+h))] \neq -[b_i(g+h)]$, as the number on the left is always nonnegative. Equations~\ref{eq:sumis2} and  \ref{eq:sumis2'}, and the exponents in the description of $\mathcal{I}$ are {\em not} taken modulo $\ell$.} an equivariant Euler class $\prod_{i\in R''_{g,h}} -b_i^2u^2.$ Thus,

\excise{We thus obtain}
\begin{align*}
\alpha_g\smile\alpha_h &= \left(\prod_{i\in R''_{g,h}, j\in R_{g,h},k\in R'_{g,h}} -b_i^2u^2(b_ju)(-b_ku)\right)\alpha_{[g+h]}.
\end{align*}

A case-by-case analysis shows that this expression agrees with the ideal $\mathcal{I}$ above. When $i$ satisfies $[b_i(g+h)]=0$, the exponents appearing in $\mathcal{I}$ are equal and either 0 or 1, depending on whether $i\in R''_{g,h}$ or not. For $[b_i(g+h)]\neq 0$, we have $$[b_ig]+[b_ih]-[b_i(g+h)]=[b_ig]+[b_ih]-(\ell-[-b_i(g+h)]) = [b_ig]+[b_ih]+[-b_i(g+h)]-\ell.$$
This sum is $\ell$ if and only if $[b_ig]+[b_ih]+[-b_i(g+h)]=2\ell$, or $i\in R_{g,h}$. Similarly, the sum $[-b_ig]+[-b_ih]-[-b_i(g+h)]=\ell$ if and only if $[-b_ig]+[-b_ih]+[b_i(g+h)]=2\ell$, or $i\in R'_{g,h}$.

Lastly we note that $\alpha_0=1$ because it is the generator of $H_{S^1}^*((\C^{n+1}\oplus \C^{n+1})^{id})$; for any class $\alpha\in NH_{S^1}^{*,\Gamma}(Y)$, we have $\alpha\smile \alpha_0=\alpha$ (note that $R_{g,id}\cup R'_{g,id}\cup R''_{g,id}=\emptyset$ for any $g$).
\end{proof}

\section{Chen-Ruan cohomology of global $T$-quotients}\label{CRcohomology}

In the case that $T$ acts on a space $Z$ locally freely, the corresponding inertial cohomology of $Z$ the equals the Chen-Ruan cohomology of the orbifold $[Z/T]$, i.e.
$$
NH_T^{*.\diamond}(Z) \cong H_{CR}^*([Z/T])
$$ where here the $*$-grading is the same for these two rings \cite{GHK}. Indeed, since the $T$ action has only finite stabilizers, $NH_T^{*,\diamond}(Z) =NH_T^{*,\Gamma}(Z)$, where $\Gamma$ is the subgroup of $T$ generated by the isotropy.
In contrast, the Chen-Ruan cohomology is not defined for Hamiltonian (or hyperhamiltonian) $T$-spaces, as these spaces always have fixed points. For this reason, inertial cohomology is a good tool to use in the symplectic and hyperk\"ahler categories. Let $Y$ be a Hamiltonian $T$ space with moment map $\Phi: Y\rightarrow \mathfrak{t}^*$. The symplectic reduction $Y/\!/T$ at a regular value $\alpha$ is defined by
$$
Y/\!/T(\alpha):= \Phi^{-1}(\alpha)/T,
$$ where we suppress $\alpha$ when it is understood. We say $Y$ is a {\em proper Hamiltonian $T$-space} if for some $\xi\in \mathfrak{t}$,  $\langle \Phi,\xi\rangle$ is a proper function on $Y$.
\excise{We will write this space $[Y/\!/T]$ to indicate its presentation as the quotient by $T$ of a {\em manifold}. In particular, the integral cohomology of $[Y/\!/T]$ may be different than the integral cohomology of the underlying topological space $Y/\!/T$.}

Inertial cohomology -- like Chen-Ruan cohomology -- is not in general functorial. A map - even an equivariant map - between spaces does not necessarily induce a {\em ring} map in the other direction on inertial or Chen-Ruan cohomology. The inclusion $\Phi^{-1}(\alpha)\hookrightarrow Y$ does not {\em a priori} induce a map of rings in inertial cohomology; however, we proved it does in this case.

\begin{theorem}[Goldin, Holm, Knutson]\label{th:GHK} Let $Y$ be a proper Hamiltonian $T$-space, with moment map $\Phi: Y \longrightarrow \mathfrak{t}^*$. Let $\alpha$ be a regular value of $\Phi$. Then the inclusion $\Phi^{-1}(\alpha)\hookrightarrow Y$ induces a surjection of rings
$$
\mathcal{K}: NH_T^{*,\diamond}(Y;\Q) \longrightarrow H_{CR}^*([\Phi\inv(\alpha)/T];\Q).
$$
Furthermore, the kernel of $\mathcal{K}$ is given in \cite{GHK}.
\end{theorem}
This theorem relies heavily on the result  \cite{Ki84} due to Kirwan that the inclusion $\Phi^{-1}(\alpha)\hookrightarrow Y$ induces a surjection of rings
\begin{equation}\label{eq:Kirwan}
H_T^*(Y;\Q)\longrightarrow H^*(\Phi\inv(\alpha)/T;\Q).
\end{equation}
This property is often referred to as {\em Kirwan surjectivity.}
We are specifically interested in the toric cases, when the quotient spaces arise as the symplectic or hyperk\"ahler reduction of a linear torus action on an affine vector space $Y$. These quotients  have large residual torus action on them, making them toric varieties in the K\"ahler case, and hypertoric varieties in the hyperk\"ahler case.

When $T \cong S^1$, the reduced space is a weighted projective space $\C P^n_{(b)}$, where $(b)$ indicates the set of weights specified by the $S^1$ action on $Y$. Theorem~\ref{th:GHK} thus provides a new proof and formula for the Chen-Ruan cohomology of toric varieties - including weighted projective spaces - over $\Q$. See also \cite{BCS} (and an explanation of the equivalence in \cite{GHK}).

The difficulty in proving a theorem analogous to Theorem~\ref{th:GHK} over $\Z$ is that Kirwan surjectivity (Equation~\ref{eq:Kirwan}) does not hold over $\Z$. However, it does in certain cases: for an effective $S^1$ action on a vector space, it can be shown that the critical set of $\Phi$ and also of $\|\Phi\|^2$ are torsion-free. This implies the following corollary:

\begin{corollary}[Corollary 9.3 \cite{GHK}]
Let $S^1$ act on $\C^{n+1}$ with positive weights. Then the ring homomorphism
$$
\mathcal{K}: NH_{S^1}^{*,\diamond}(\C^{n+1};\Z) \longrightarrow H_{CR}^*([\C P^{n}_{(b)}];\Z)
$$ is a surjection.
\end{corollary}

In the hyperk\"ahler case, $Y$ has three K\"ahler structures; if the $T$ action is hyperhamiltonian, then there are three moment maps, one for each K\"ahler structure. They may be combined into the two maps $\Phi_\R:Y\rightarrow \mathfrak{t}$ and $\Phi_\C:Y\rightarrow \mathfrak{t}\otimes \C$, as explained in \excise{REFERENCE KONNO IN PROCEEDINGS}\cite{Ko}. The {\em hyperk\"ahler reduction} at a regular value $(\nu_1,0+0i)$ is given by $\Phi_\R^{-1}(\nu_1)\cap\Phi_\C^{-1}(0)/T$, and is denoted $X/\!/\!/\!/T$, sometimes with a subscript to indicate the point of reduction.

For hyperk\"ahler quotients, there is also a well-defined Kirwan map induced from the inclusion $\Phi_\R^{-1}(\nu)\cap\Phi_\C^{-1}(0)\hookrightarrow Y$, though the induced map
\begin{equation}\label{eq:HKKirwan}
H_T^*(Y;\Q)\rightarrow H_T^*(\Phi_\R^{-1}(\nu)\cap\Phi_\C^{-1}(0);\Q)
\end{equation}
is not known to be surjective. It also has not been proven that a ring map exists from the inertial cohomology of a hyperk\"ahler manifold $Y$ to the Chen-Ruan cohomology of its hyperk\"ahler reduction. However, in the case that $Y=T^*\C^{n+1}$ with a linear $T$ action, surjectivity  of Equation (\ref{eq:HKKirwan}) is known to hold \cite{Ko}.
 This allowed the author and M. Harada to prove the following:
\begin{theorem}[Goldin-Harada]\label{th:GH} Let $Y=T^*\ C^{n+1}$ with a $T$ action given by acting on $\C^{n+1}$ and by its inverse on the fiber directions, as specified in \excise{REFERENCE KONNO IN PROCEEDINGS}\cite{Ko}. Let $\Phi_\R\oplus\Phi_\C$ be a hyperhamiltonian moment map for this action. Then the inclusion $\Phi_\R^{-1}(\alpha)\cap\Phi_\C^{-1}(0)\hookrightarrow Y$ induces a surjection
\begin{equation}\label{eq:hypsur}
NH_T^{*,\diamond}(Y;\Q)\longrightarrow H_{CR}^*([\Phi_\R^{-1}(\alpha)\cap\Phi_\C^{-1}(0)/T];\Q).
\end{equation}
The ring $NH_T^{*,\diamond}(Y)$ and the kernel of (\ref{eq:hypsur}) are computed in \cite{GH}.
\end{theorem}

This theorem led the authors to a combinatorial description of the Chen-Ruan cohomology of hypertoric varieties. Another description was independently discovered by \cite{JT} using the language of stacks and fans.

\excise{Our goal is then to do the same analysis on the topology of the moment maps for the hypertoric case. In this case, torsion can arise where.... and compare to the techniques developed by Konno to make the moment map equivariantly perfect.
[And I think the following is repetitive.]
\begin{theorem}[Goldin, Harada] Over $\Q$ there is a surjection via $S^1$ to the Chen-Ruan cohomology of $T^*\C^n/\!/\!/\!/T$.
\end{theorem}
}

\excise{
\section{Construction of hypertoric varieties}\label{se:hypertoric}
We follow Konno's quotient construction of hypertoric varieties. Let $\H^n$ be the quaternionic vector space with three complex structures, denoted $J_1,J_2,J_3.$ We may identify $\H^n \iso T^*\C^n = \C^n\times \C^n$ by
$$
(q_1,\dots, q_b) \longmapsto ((z_1,\dots, z_n),(w_1,\dots, w_n)),
$$
where $q_i = z_i + w_iJ_2,$ for $i=1,\dots, n$, with $z_i, w_i\in \R+\R J_1\iso \C$. We use this identification to write $(z,w)\in \H^n$, where $z$ and $w$ may be thought of as $n$-vectors in $\C^n$. Correspondingly, there are three K\"ahler structures, coming from compatibility with the complex structures and the metric  on $\H^n$.  Let $T^n =\{{\bf t}=(t_1,\dots, t_n)\in \C^n: |t_i|=1\}$. There is a natural action by $T^n$ on $\H^n$ on the right (Hamiltonian with respect to all three symplectic forms), given by
$$
{\bf t}\cdot (z,w) := (zt, wt^{-1}).
$$
From this action and the three K\"ahler forms, we obtain three moment maps, which we denote $\phi_i: \H^n \rightarrow (\t^n)^*, i=1,2,3$. These may be written as a real moment map and a complex moment map. Let $\{\varepsilon_i\}$ be a basis for $T^n$ given by the splitting above. Let $\{u_i\}$ be the dual basis of $(\t^n)^*$. The explicit formula for the moment maps are
\begin{align*}
\phi_\R(z,w):=\phi_1(z,w) &= \pi \sum_{i=1}^n (|z_i|^2 - (w_i|^2)u_i\\
\phi_\C(z,w):=(\phi_2+\sqrt{-1}\phi_3)(z,w) &= -2\pi\sqrt{-1}\sum_{i=1}^n z_iw_iu_i.
\end{align*}

Lastly, let $T^k\subset T^n$ be a compact subtorus, with Lie algebra $\t^k$.
This inclusion induces an exact sequence
\begin{equation}\label{eq:Delzant-Lie}
\xymatrix @R=-0.2pc {
0 \ar[r] & \t= \t^k \ar[r]^{\iota} & \t^n \ar[r]^{\beta} & \t^d \ar[r] & 0, \\
         &                                & \varepsilon_i \ar@{|->}[r] & a_i & \\
}
\end{equation}
 and its dual sequence
\begin{equation}\label{eq:Delzant-Lie-dual}
\xymatrix @R=-0.2pc {
0 \ar[r] & (\t^d)^* \ar[r]^{\beta^*} \ar[r] & (\t^n)^* \ar[r]^{\iota^*} & \t^* = (\t^k)^* \ar[r] & 0.\\
     & & u_i \ar@{|->}[r] & \lambda_i := \iota^* u_i & \\
}
\end{equation}
where $T^d = T^n/T^k$ be the quotient torus, and $\t^d$ is its Lie algebra. The maps $\psi_i:= i^*\circ \phi_i: \H^n\rightarrow (\mathfrak{t}^k)^*$ are the three moment maps for the $T^k$ action on $\H^n$. We assume that the restriction of the $T^k$ action on the first $n$ components has a symplectic (real) moment map that is proper and bounded below in one direction.\footnote{This ensures that the corresponding toric variety is compact.} Let $\Psi := (\psi_1, \psi_2, \psi_3)$.
\excise{
Then Konno proved \cite{Ko} the following lemma.
\begin{lemma}
An element $(\nu_1,\nu_2,\nu_3)\in (\mathfrak{t}^k)^*\oplus (\mathfrak{t}^k)^*\oplus (\mathfrak{t}^k)^*$ is a regular value of $\Psi$ if an only if for any $J\subset\{1,\dots, n\}$ with $|J|<k$, not all of $\nu_1,\nu_2,\nu_3$ are contained in the span of $\{i^*u_j| j\in J\}$.
\end{lemma}
}
We may now define a {\bf hypertoric variety} as the quotient of a level set of a hyperk\"ahler moment map by the torus $T^k$.
\begin{definition}
Let $T^k$ act on $T^*\C^n$ as specified above, with hyperk\"ahler moment map $\Psi$. Suppose that $\nu = (\nu_1,0,0)$ is a regular values of $\Psi$. Then
$$
\mathfrak{M}=T^*\C^n/\!/\!/\!/_{\nu}T := \Psi^{-1}(\nu)/T
$$ is a {\bf hypertoric variety.}
\end{definition}
$\mathfrak{M}$ has an induced hyperk\"ahler structure on it, as well a residual action by $T^d\iso T^n/T^k$. In the case that the action of $T^k$ on $\Psi^{-1}(\nu)$ is {\em free}, this quotient is a manifold and is termed a ``toric hyperk\"ahler manifold" by Konno \cite{Ko}. However, since the space is not actually toric, we avoid this name. For a generic regular value, the action is only {\em locally free}, implying that the quotient is an orbifold.

\excise{
A description of the Chen-Ruan cohomology of these spaces over $\Q$ can be found in \cite{GH}. The reason that the results and techniques in that paper cannot be immediately extended over $\Z$ is that Kirwan surjectivity is not known to hold over $\Z$.}

\begin{mainexample}
Let $S^1$ act on $T^*\C^{n+1}$ with weights positive weights $b_0,\dots, b_n$ as in Example~\ref{mex:inertial}. Then for any regular value $\nu = (\nu_1,0,0)\in (\mathfrak{t}^*)^3$, we term the corresponding hypertoric variety $\mathfrak{M} =\Psi^{-1}(\nu)/S^1$ a {\bf weighted hyperprojective space.}
\end{mainexample}

\excise{We denote this hypertoric variety $T^*\C P^{n}_{(b)}$, where $b=(b_0,\dots, b_n)$, and call it a {\bf weighted hyperprojective space}. MAYBE IT SHOULDN'T BE DENOTED THIS WAY BECAUSE IT'S A BIT OF A CONSEQUENCE, NOT BY DEFINITION, THAT THIS IS A COTANGENT BUNDLE.}
}
\excise{
\section{The core of a hypertoric variety}\label{se:core}

There is a $\C^\times$ action on $T^*\C^{n+1}$ that rotates the fibers, i.e. for $\lambda\in \C^\times$,
$$
\lambda\cdot (z_0,\dots, z_n, w_0,\dots, w_n)= (z_0,\dots, z_n, \lambda w_0,\dots, \lambda w_n).
$$
This action descends to any hypertoric quotient $\mathfrak{M}$, since the $\C^\times$ action commutes with any $T$ action on $T^*\C^{n+1}$. Furthermore, the action of $S^1\subset \C^\times$ on $\mathfrak{M}$ is Hamiltonian with respect to the first of its symplectic forms. The map $\mu ([z,w]) = \frac{1}{2}|w|^2$ is a corresponding moment map for this $S^1$ action. Then $\mu$ is a proper map and $\mu^{-1}(0)=\C^{n+1}/\!/_{\nu_1}T$ \cite{HP}. Furthermore, $\mu$ is a an orbifold Morse-Bott function in the sense of \cite{LT}. The fact that $\mu$ is proper ensures that $\lim_{\lambda\rightarrow 0} \lambda\cdot p$ exists for all $p\in \mathfrak{M}$, allowing us to define the {\em core} of a hypertoric variety.
\begin{definition}
Let $\mathfrak{M}$ be a hypertoric variety (with at worst orbifold singularities). Let
$$
C(\mathfrak{M}) := \{p\in \mathfrak{M}: \lim_{\lambda\rightarrow \infty} \lambda\cdot p\mbox{ exists.}\}
$$
The set $C(\mathfrak{M})$ is called the {\bf core} of the hypertoric variety $\mathfrak{M}$.
\end{definition}

The core is a finite union of toric varieties \cite{Ko}. Let $F$ be a connected component of the fixed point set $\mathfrak{M}^{\C^\times}$, and $U_F$ the closure of the set of points that flow into $F$ as $\lambda\to\infty$. Then
$$C(\mathfrak{M}) = \cup_{F\subseteq \mathfrak{M}^{\C^\times}} U_F.$$

The toric variety $\C^{n+1}/\!/_{\nu_1}T$ is always in the core, since the $\C^\times$ orbits of points in $\mu^{-1}(0)=\C^{n+1}/\!/_{\nu_1}T$ must stay in this set as $\mu$ is $S^1\subset \C^\times$ invariant, and all such points have existing limits since $\C^{n+1}/\!/_{\nu_1}T$ is compact.  The core is the union of the toric varieties corresponding (in the sense of \cite{LT}) to the convex polytopes appearing in a hyperplane arrangement associated to the hypertoric variety $\mathfrak{M}$.
Note that $\mathfrak{M}$ has a residual action of a $T^d\iso T^{n+1}/T^k$ torus, where $T^k$ is the torus by which we took the quotient. There is a very strong sense in which $\mathfrak{M}$  is homotopic to its core.
\begin{lemma}[Bielawski-Dancer] The core $C(\mathfrak{M})$ is (an equivariant) deformation retract of $\mathfrak{M}$.
\end{lemma}

In the case of reduction of $T^*\C^{n+1}$ by $S^1$, the image of the (real) moment map is described by exactly $n+1$ (oriented) hyperplanes in $n$-dimensional space. Their intersections  form maximally one connected convex polytope, and it is a simplex. Thus when $\mathfrak{M}$ is a weighted hyperprojective space, the weighted projective space $\C^{n+1}/\!/_{\nu_1}S^1$ {\em is} the core\excise{. Indeed, $\mathfrak{M}$ retracts to this weighted projective space }\cite{HP}.
\excise{\em THERE MUST BE A NICER WAY OF SEEING THIS. USE NICK'S THESIS, PAGE 27 IN PARTICULAR.
SHOW THAT IN THIS CASE, $\mathfrak{M}=T^*\C P^n_{(b)}$ AND $\mathfrak{L} = \C P^n_{(b)}$.

ALSO IT MIGHT BE A GOOD IDEA TO WRITE THIS IN TERMS OF KONNO'S WORK INSTEAD OF NICK'S. KONNO MIGHT EVEN STATE EXPLICITLY THIS FACT ABOUT THE $S^1$ CASE, SO WE WON'T HAVE TO PROVE IT.}

From this lemma, it may seem immediate to conjecture that the Chen-Ruan cohomology of a weighted hyperprojective space equals the Chen-Ruan cohomology of the corresponding weighted projective space. However, while the homotopy is equivariant, it does not preserve all the normal bundles that come into play in the Chen-Ruan (or inertial) cohomology. It is not a {\em $S^1$-representation homotopy} in the sense of Definition~\ref{de:representationhomotopy}. As we shall see by a direct comparison with the description of the CR-cohomology of these weighted projective spaces presented in \cite{Ho}, the rings have a different product structure, even over $\Q$-coefficients.
}

\section{The Chen-Ruan cohomology of weighted hyperprojective spaces}\label{se:computationCRweightedhyperprojective}

Let $S^1$ act on $T^*\C^{n+1} \cong \C^{n+1}\oplus \C^{n+1}$ with weights $b_0,\dots, b_n\in \Z$ on the first copy of $\C^{n+1}$ and with weights $-b_0,\dots, -b_n$ on the second copy. Consider the homomorphism $\phi: S^1\rightarrow T^{n+1}\cong S^1\times\cdots\times S^1$ given by $t\to (t^{b_0},\dots, t^{b_n})$. This induces an inclusion of Lie algebras $\iota: \mathfrak{s}^1\hookrightarrow \mathfrak{t^{n+1}}$. Let $\{\varepsilon_i\}$ be a basis for $\mathfrak{t}^{n+1}$, and $\{u_i\}$ a basis for its dual $(\mathfrak{t}^{n+1})^*.$ Let $T^n\cong T^{n+1}/S^1$. We define the vectors $\{a_i\}$ by the image of $\{\varepsilon_i\}$ in the exact sequence
\begin{equation}\label{eq:Delzant-Lie}
\xymatrix @R=-0.2pc {
0 \ar[r] & \t= \mathfrak{s}^1 \ar[r]^{\iota} & \t^{n+1} \ar[r]^{\beta} & \t^n \ar[r] & 0, \\
         &                                & \varepsilon_i \ar@{|->}[r] & a_i & \\
}
\end{equation}
\excise{
and and the vectors $\lambda_i$ by
\begin{equation}\label{eq:Delzant-Lie-dual}
\xymatrix @R=-0.2pc {
0 \ar[r] & (\t^d)^* \ar[r]^{\beta^*} \ar[r] & (\t^n)^* \ar[r]^{\iota^*} & \t^* = (\t^k)^* \ar[r] & 0.\\
     & & u_i \ar@{|->}[r] & \lambda_i := \iota^* u_i, & \\
}
\end{equation}
}

\begin{theorem}\label{th:CRhyperprojective}
The (rational) Chen-Ruan cohomology of the weighted hyperprojective space $\mathfrak{M}=T^*\C^{n+1}/\!/\!/\!/_\nu S^1$ for any regular value $\nu$ is
$$
H_{CR}^*([\mathfrak{M}]) \cong \Q[u_0,\dots u_n,\alpha_0,\dots, \alpha_{\ell-1}]/\mathcal{I}+\mathcal{J}+\mathcal{K},
$$
where $\mathcal{I}, \mathcal{J}$ and $\mathcal{K}$ are given by
\begin{align*}
\mathcal{I}&=
\langle\alpha_g\smile\alpha_h-
\alpha_{[g+h]}\prod_{i=0}^n (b_iu)^{\frac{[b_ig]+[b_ih]-[b_i(g+h)]}{\ell}}
\cdot (-b_iu)^{\frac{[-b_ig]+[-b_ih]-[-b_i(g+h)]}{\ell}}\rangle\\
\mathcal{J} &= \langle image(\beta^*)\rangle\quad\mbox{and}\\
\mathcal{K} &= \sum_{g\in \Gamma}\langle \alpha_g\prod_{c_i\neq 0} b_iu_i |\ \mbox{there exist $c_j\in \R$ such that }\sum_{j=0}^n c_j\varepsilon_j\in image(\iota)-0\rangle.
\end{align*}
\excise{and
\begin{align*}
\mathcal{I} =
\langle
\bigcup_{g\in \Gamma}\prod_{i\in A_g} -b_i^2u^2
\rangle
\end{align*}
where $A_g = \{i\in\{0,\dots, n\}|\ g\in \langle b_i\rangle\leq \Z/\ell\Z\}$.}
where $\iota$ is given in by the inclusion in (\ref{eq:Delzant-Lie}) and $\beta^*$ is the dual map $(\mathfrak{t}^n)^*\rightarrow (\mathfrak{t}^{n+1})^*$ to $\beta$. See  \excise{REFERENCE KONNO IN PROCEEDINGS}\cite{Ko} for more on the construction of hypertoric varieties.
\footnote{The numbers $b_i$ are not relevant to the ideal $\mathcal{K}$; here they serve only as a reminder of a conjectural answer to the question posed in Remark~\ref{re:openquestion}.}
\end{theorem}
\begin{proof}
Recall that $\Gamma\iso \Z/\ell \Z$, where $\ell=lcm(b_0,\dots, b_n)$. The $\Gamma$-subring of the inertial cohomology $NH_{S^1}^{*,\diamond}(T^*(C^{n+1}))$ is given by $\Q[u,\alpha_0,\dots, \alpha_{\ell-1}]/\mathcal{I}$, as shown in Example~\ref{mex:inertial}. Note that this ring is isomorphic to $\Q[u_0,\dots, u_n,\alpha_0,\dots, \alpha_{\ell-1}]/\mathcal{I}+\mathcal{J}$ where $\mathcal{J}$  consists of the linear relations among the $u_i$ given by killing  $\langle image(\beta^*)\rangle$. We need only compute the kernel of the surjective map described in Theorem~\ref{th:GH}. By \cite{GH}, the kernel is generated by the kernel of $H_T^*((T^*\C^{n+1})^g)\rightarrow H^*((T^*\C^{n+1})^g/\!/\!/\!/T)$ for each $g\in \Gamma$. This computation is done in \cite{HS}, where they found that the kernel for each $g\in \Gamma$ is given by the product of those $u_i$ such that the intersection of the corresponding hyperplanes perpendicular to $a_i$ is empty. Because the torus may act noneffectively on $(T^*\C^{n+1})^g$, this is equivalent to the ideal given in (\cite{GH}, Equation 5.6).
Here we express this condition as Konno does in \cite{Ko}, by the product of those $u_i$ such that the nonzero sum of the corresponding vectors is in $image(\iota)$.
\end{proof}

\excise{the equivariant Euler class of the normal bundle to the fixed point set $\left((T^*\C^{n+1})^g\right)^T$ for each $g\in \Gamma$. This set is the point $\{0\}$ for all $g$. Thus, for any $g\in \Gamma$, there is one contribution to the kernel is given by the $S^1$-equivariant Euler class of $(T^*\C^{n+1})^g$. Note that
$$
(T^*\C^{n+1})^g = \bigoplus_{i\in A_g} \C_i\oplus \C_i,
$$
where $A_g$ is the set of $i$ such that $g\in\langle b_i\rangle$, where this is the subgroup in $\Z/\ell\Z$ generated by $b_i$.
The $S^1$ action on each plane $\C_i\oplus \C_i$ is given by $t\cdot(z_i,w_i)=(t^{b_i}z_i, t^{-b_i}w_i)$.
The equivariant Euler class of this trivial vector bundle is $\prod_{i\in A_g} -b_i^2u^2$.}

\begin{remark}
The Chen-Ruan cohomology of a weighted hyperprojective space is different from that of its core, a weighted projective space --  even over $\Q$. For example, in the case that $S^1$ acts on $T^*\C^3$ with $(b)=(2,1,1)$, the core $C(\mathfrak{M})$ is a weighted $\C P^2$ with exactly one singular point whose isotropy is $\Z/2\Z$. The Chen-Ruan cohomology of $[\C P^2_{(2,1,1)}]$ according to \cite{GHK} is
$\Q[u,\alpha]/\langle u^3, u\alpha, \alpha^2-u^2\rangle$ with degree $\deg u =2$, and $\deg \alpha=2$.
In contrast, the Chen-Ruan cohomology of $[\mathfrak{M}]\excise{=[T^*\C P^2_{(b)}]}$ is $\Q[u,\gamma]/\langle u^3, \gamma^2, u\gamma\rangle$ with $\deg u=2$ and $\deg \gamma=4$. This computation is done in \cite{GH}, \S~6.
\end{remark}

\begin{remark}\label{re:openquestion}
A natural (open) question is whether the Chen-Ruan cohomology of weighted hyperprojective spaces can be computed over $\Z$. Indeed, Theorem~\ref{th:CRhyperprojective} suggests what the answer should be. The inertial cohomology of $T^*\C^n$ computed in Example~\ref{mex:inertial} is over $\Z$, and the ideals in Theorem~\ref{th:CRhyperprojective} are expressed with integers arising from the weights. Indeed, a proof of such a formula would follow from the results in \cite{GH} if Kirwan surjectivity holds over $\Z$ for $S^1$ quotients of vector spaces. It is possible that Konno's techniques \cite{Ko} could be sufficient to prove this result, if one generalizes the line bundles he constructs over hypertoric manifolds to orbi-bundles over hypertoric orbifolds. It may also be possible to analyze directly the Morse theory in this case, to prove that torsion does not obstruct surjectivity.
\end{remark}
\section{Homotopy Invariance of Inertial Cohomology}\label{se:homotopy}

Inertial cohomology, like Chen-Ruan cohomology, is not a cohomology theory. An equivariant map of $T$-spaces, $f:X\rightarrow Y$ does not generally induce a map (of rings) $NH_T^{*,\diamond}(Y)\rightarrow NH_T^{*,\diamond}(X)$. Of course, a map in equivariant cohomology $H_T^*(Y^g)\rightarrow H_T^*(X^g)$  exists for all $g\in T$, so there is a map of graded vector spaces in inertial cohomology. It is the product structure that fails to respect functoriality.

A simple example of this failure is as follows. Let $S^1$ act on $X=\C$ with weight 2 and let $f:X\rightarrow pt$ map every point in $X$ to $Y=pt$. As vector spaces, $$NH_T^{*,\diamond}(X)=\bigoplus_{g=\pm 1}H_T^*(\C)\oplus\bigoplus_{g\in T, g\neq \pm 1} H_T^*(\{0\})\quad \mbox{and}\quad NH_T^{*,\diamond}(Y)=\bigoplus_{g\in T} H_T^*(pt),$$ with the natural map induced in equivariant cohomology $f^*:H_T^*(Y^g)\rightarrow H_T^*(X^g)$ given by the identity map for each $g\in T$. The product on $NH_T^{*,\diamond}(Y)$ is just the usual product on $H_T^*(pt)$ shifted by the grading of the group element. Thus in $NH_T^{*,\diamond}(Y)$, $1_g\smile 1_h = 1_{gh}$ for all $g,h\in S^1$, where the subscript indicates the $T$-grading.
In $NH_T^{*,\diamond}(X)$, however, an obstruction bundle may play a role.
 Choose $g=h=(gh)^{-1}=e^{\frac{2\pi i}{3}}$. These elements fix $\{0\}$ alone. Since the action on the normal bundle $\C$ is with weight 2, we have $a(g)=a(h)=a((gh)^{-1})=\frac{2}{3}$, implying that $E=\C$ is the obstruction bundle. The equivariant Euler class of this bundle is $2u$, where $u$ is the (positive) generator of $H_T^*(pt;\Z)$. It follows that $1_g\smile 1_h = (2u)_{gh}.$
Thus $f^*$ does not induce a ring map on the inertial cohomology. It is  clear that requiring the homotopy to be $T$-equivariant does not fix the problem.

We can avoid this ``change in obstruction bundle" by insisting that maps preserve (in an appropriate sense) the normal bundles to the fixed point sets. Such maps induce ring maps in inertial cohomology. When the maps are also homotopy equivalences, we obtain an isomorphism in inertial cohomology.\excise{ We make this precise below.}

\excise{We begin by noting that a smooth homotopy $F:[0,1]\times X\rightarrow Y$
between two $T$-equivariant maps $f, f':X\rightarrow Y$ is a {\em $T$-equivariant homotopy} \excise{if $F(0,x)=f(x)$, $F(1,x)=f'(x)$ for all $x\in X$, and additionally} if for all $g\in T$ and $s\in [0,1]$, $F(s, g\cdot x)=g\cdot F(s,x)$. We call $f$ and $f'$ T-equivariantly homotopic maps.} We say that the $T$-spaces $X$ and $Y$ are {\em $T$-equivariantly homotopic} if there exist equivariant maps $f:X\rightarrow Y$ and $e:Y\rightarrow X$ such that $f\circ e$ is $T$-equivariantly homotopic to $id_Y$ and $e\circ f$ is $T$-equivariantly homotopic to $id_X$. In this case, the map $f:X\rightarrow Y$ is said to be an {\em equivariant homotopy equivalence.} Note that if $X$ and $Y$ are $T$-equivariantly homotopic, then so are $X^g$ and $Y^g$ for all $g\in T$.

Lastly, we recall some notions of equivalence of vector bundles. Let $f:X\rightarrow Y$ be a smooth map, and $E$ a vector bundle over $Y$. Let $f^*(E)\rightarrow X$ denote the pullback of $E$ to $X$. We say that a $T$-bundle $E\rightarrow Y$ is isomorphic to a $T$-bundle $E'\rightarrow X$ if there exists a $T$-equivariant homotopy equivalence $f:X\rightarrow Y$ such that $f^*(E)\iso E'$ as $T$-bundles over $X$.

\begin{definition}\label{de:representationhomotopy}
Let $T$ act on manifolds $X$ and on $Y$, and let $F:[0,1]\times X\rightarrow Y$ be a smooth $T$-equivariant homotopy between smooth equivariant maps $f:X\rightarrow Y$ and $f':X\rightarrow Y$, such that $F(s,X)$ is a submanifold of $Y$ for all $s\in [0,1]$. We say that $F$ is a {\bf $T$-representation homotopy} if, for all $g,h\in T$ and all $k\in \langle g,h\rangle$, the $T$-bundles $\nu(F(s,X)^{g,h}\subset F(s,X)^k)$ are isomorphic for all $s\in [0,1]$. In this case we write $f\sim_{T} f'$.\footnote{
Note that, for all $s\in [0,1]$, $F(s,X)^g$ is a smooth submanifold of $F(s,X)$.
} We say that $X$ and $Y$ are {\bf T-representation homotopic spaces} if there exist smooth equivariant maps $f:X\rightarrow Y$ and $e:Y\rightarrow X$ such that $f\circ e\sim_{T} id_Y$ and $e\circ f\sim_{T} id_X$.
\end{definition}

\begin{remark} Note that, for $S^1$ acting on $X=\C$ with  weight 2 and on $Y=pt$ trivially, the contraction map is an $S^1$-equivariant homotopy from $X$ to $Y$, but there is no $T$-representation homotopy from $X$ to $Y$. On the other hand, if $X=\C$ with a weight 1 action, and $Y = \C^3$ with a weight 1 action on one copy of $\C$ and a trivial action on a copy of $\C^2$, then $X$ and $Y$ are $S^1$-representation homotopic, though they are not (equivariantly) homeomorphic. Similarly, if $X=\R^2-\{0\}$ with a rotating $S^1$ action, and $Y=S^1=\{x\in X: |x|^2=1\}$ with $S^1$ acting by the restriction of how it acts on $X$, then there is a $T$-representation homotopy equivalence from $X$ to $Y$, though the spaces are not homeomorphic.
\end{remark}

\excise{With this definition in place, we prove the following theorem:}
\begin{theorem}\label{th:homotopy}
Let $X$ and $Y$ be $T$-representation homotopic spaces. Then there is a ring isomorphism
$$
NH_T^{*,\diamond}(X;\Z)\cong NH_T^{*,\diamond}(Y;\Z).
$$
Indeed, if the homotopy equivalence is given by the maps $f:X\rightarrow Y$ and $f':Y\rightarrow X$, then  $f^*:NH_T^{*,\diamond}(Y;\Z)\rightarrow NH_T^{*,\diamond}(X;\Z)$ induced by $f$ is an isomorphism.
\end{theorem}

\begin{proof}
It is immediate that $f:X\rightarrow Y$ induces an isomorphism of groups $f^g: H_T^*(Y^g;\Z)\rightarrow H_T^*(X^g;\Z)$ for all $g\in T$, since $f$ is a $T$-equivariant homotopy equivalence and equivariant cohomology is a homotopy invariant. This immediately implies that $f$ induces an isomorphism of groups
$f^*:=\bigoplus_{g\in T} f^g: NH_T^{*,\diamond}(Y;\Z)\rightarrow NH_T^{*,\diamond}(X;\Z).$
We now check that $f^*$ is  an isomorphism of rings. Fix $g,h$ and $k\in\langle g,h\rangle$, and let $\nu_X := \nu(X^{g,h}\subset X^k)$, and $\nu_Y:= \nu(Y^{g,h}\subset Y^k)$. We identify $\nu_X$ with the complement of the tangent bundle $T(X^{g,h})$ in $T^*(X^k)$ via a $T$-invariant metric, and similarly for $\nu_Y$. Since $f'\circ f \sim_T id_X$ is a T-representation homotopy, the map $d(f'\circ f)|_{\nu_X}: \nu_X\rightarrow \nu_X$ is a $T$-equivariant isomorphism. Thus $df|_{\nu_X}:\nu_X \rightarrow TY$ is injective. By equivariance of $f$, the image of $df|_{\nu_X}$ lies in $\nu_Y$ restricted to $f(X)$. It now follows that $\nu_X \subset f^*\nu_Y$, i.e.  $\nu(X^{g,h}\subset X^k)$ is isomorphic to a subbundle of $\nu(Y^{g,h}\subset Y^k)$ restricted to $f(X)$. Similarly, $f\circ f'\sim_T id_Y$ implies that $\nu(Y^{g,h}\subset Y^k)$ occurs as a subbundle of $\nu(X^{g,h}\subset X^k)$ restricted to $f'(X)$ for all $g,h\in T$ and all $k\in \langle g,h\rangle$. Since $f|_{X^{g,h}}:X^{g,h}\rightarrow Y^{g,h}$ and $f'_{Y^{g,h}}:Y^{g,h}\rightarrow X^{g,h}$ are homotopy equivalences, \excise{for all $g,h$, we conclude that}
$\nu(X^{g,h}\subset X^k)\cong \nu(Y^{g,h}\subset Y^k)\quad\mbox{for all $k\in \langle g,h\rangle$}.$ We assume that $Y^{g,h}$ is connected (and otherwise make the same argument for connected components). The obstruction bundle $E_Y|_{Y^{g,h}}$ is a subbundle of $\nu(Y^{g,h}\subset Y=Y^{id})$ composed of those isotypic components $I_\lambda$ for which $a_\lambda(g)+a_\lambda(h)+a_\lambda((gh)^{-1})=2$. Similarly, the obstruction bundle $E_X|_{X^{g,h}}$ in the ($\langle g,h\rangle$-equivariantly) isomorphic vector bundle $\nu(X^{g,h}\subset X)$ consists of those isotypic components $I'_{\lambda'}$ for which $a_{\lambda'}(g)+a_{\lambda'}(h)+a_{\lambda'}((gh)^{-1})=2$. Since the representations of $\langle g,h\rangle$ on each fiber are isomorphic, the obstruction bundles are isomorphic. Note that these bundles are $T$-equivariantly isomorphic, because the total normal bundles are.
Denote the equivariant Euler class of $E_X|_{X^{g,h}}$ by $\epsilon_X$ and that of $E_Y|_{Y^{g,h}}$ by $\epsilon_Y$.

The isomorphism of equivariant bundles implies that,
under the map $f^*:H_T^*(Y^{g,h};\Z)\rightarrow H_T^*(X^{g,h};\Z)$, we have $f^*(\epsilon_X)=\epsilon_Y.$  This in turn implies that for all $a\in NH_T^{*,g}(Y)$ and $b\in NH_T^{*,h}(Y)$, we have
\begin{align*}
f^*(a\smile b)&= f^{*}[(\overline{e}_{3,Y})_*(e_{1,Y}^*(a)\cdot e_{2,Y}^*(b)\cdot\epsilon_Y)]\\
&= (\overline{e}_{3,X})_*[f^{*}(e_{1,Y}^*(a)\cdot e_{2,Y}^*(b)\cdot \epsilon_Y)]\\
&= (\overline{e}_{3,X})_*[e_{1,X}^*(f^*(a))\cdot e_{2,X}^*(f^*(b))\cdot f^{*}(\epsilon_Y)]\\
&= (\overline{e}_{3,X})_*[e_{1,X}^*(f^*(a))\cdot e_{2,X}^*(f^*(b))\cdot (\epsilon_X)]\\
&= f^*(a) \smile f^*(b),
\end{align*} as desired.\end{proof}

Chen-Ruan cohomology is not obviously a representation homotopy invariant in the same way that inertial cohomology is. A generic orbifold is described locally by isotropy groups and their representations at particular points in the orbifold (see \cite{CR}). However, these groups are not naturally subgroups of one larger group. To define a similar concept of maps between orbifolds, one must work with local charts (upstairs) and their gluing maps. However, a map between orbifolds $f: [X]\rightarrow [Y]$ may not induce a well-defined pull-back orbibundle $f^*E$ for an orbibundle $E$ over $[Y]$. See \cite{CR1} for more details on {\em good maps}. However, one immediate consequence of Theorem~\ref{th:homotopy} is the following corollary.
\begin{corollary}
Let $T$ act on a stably complex spaces $X$ and $Y$ with finite isotropy. If $X$ and $Y$ are $T$-representation homotopic, then $H_{CR}^*([X/T];\Z)\iso H_{CR}^*([Y/T];\Z).$
\end{corollary}

\excise{
THIS IS HOW TO MAKE SOME COOL TABLES AND HOW TO INSERT FIGURES
\begin{table}[ht]
\caption{}\label{eqtable}
\renewcommand\arraystretch{1.5}
\noindent\[
\begin{array}{|c|c|c|}
\hline
&{-\infty}&{+\infty}\\
\hline
{f_+(x,k)}&e^{\sqrt{-1}kx}+s_{12}(k)e^{-\sqrt{-1}kx}&s_{11}(k)e^
{\sqrt{-1}kx}\\
\hline
{f_-(x,k)}&s_{22}(k)e^{-\sqrt{-1}kx}&e^{-\sqrt{-1}kx}+s_{21}(k)e^{\sqrt
{-1}kx}\\
\hline
\end{array}
\]
\end{table}

\begin{figure}[tb]
\blankbox{.6\columnwidth}{5pc}
\caption{This is an example of a figure caption with text.}
\label{firstfig}
\end{figure}

\begin{figure}[tb]
\blankbox{.75\columnwidth}{3pc}
\caption{}\label{otherfig}
\end{figure}

}

\bibliographystyle{amsalpha}

\end{document}